\documentclass[12pt]{article}
\usepackage{graphicx}
\usepackage{epstopdf}
\usepackage{amsfonts,amsmath,times,amsthm,amssymb}
\usepackage{color, soul}
\numberwithin{equation}{section}
\newtheorem{theorem}{Theorem}[section]
\newtheorem{lemma}{Lemma}[section]
\newcommand{\Prob}{\mathbb{P}}
\newcommand{\convD}{\, \overset{D}{\longrightarrow} \,}

\newcommand{\E}{\mathbb{E}}

\newcommand{\polya}{{P\'olya}}

\newcommand{\Exp}{{\rm Exp}}
\newcommand{\matA}{{\bf A}}
\newcommand{\matB}{{\bf B}}
\newcommand{\matM}{{\bf M}} 
\newcommand{\matH}{{\bf H}} 
\DeclareRobustCommand{\stirling}{\genfrac\{\}{0pt}{}}
\begin{document}
\centerline{RESEARCH ARTICLE}
\vskip 3mm

\noindent Characterizations of asymptotic distributions of continuous-time \polya\ processes
\vskip 3mm

\vskip 5mm
\begin{table}[ht]
	\begin{tabular}{l l}
		Chen Chen & Panpan Zhang
		\\ Department of Statistics & Department of Statistics
		\\ Pennsylvania State University & University of Connecticut
		\\ State College, PA \ 16801 & Storrs, CT \ 06269
		\\ cuc19@psu.edu  & panpan.zhang@uconn.edu
	\end{tabular}
\end{table}

\vskip 3mm

\noindent ABSTRACT

We propose an elementary but effective approach to studying a general class of Poissonized tenable and balanced urns on two colors. We characterize the asymptotic behavior of the process via a partial differential equation that governs the process, coupled with the method of moments applied in a bootstrapped manner. We show that the limiting distribution of the process underlying the Bagchi-Pal urn is gamma. We also look into the tenable and balanced processes associated with randomized replacement matrix. Similar results carry over to the process, with minor modifications in the methods of proof, done mutatis mutandis.
\vskip 4mm

\section{Introduction}
\label{Sec:Intro}
The theory of urn models has received increased attention and intensive research from probabilists and statisticians owing to its conceptual simplicity and versatility. In modern times, urn models have been recognized as a fundamental and powerful mathematical tool, and gained rising popularity among researchers of applied sciences. The applications of urn models span in a wide range of areas, such as computer science~\cite{Bagchi}, clinical trails~\cite{Wei1979}, epidemiology~\cite{Egg} and physics~\cite{Ehrenfest}. We refer the interested readers to~\cite{Johnson} and~\cite{Mahmoud} for the history and applications of urn models.

In this note, we focus on \polya-type urns. A \emph{two-color \polya\ urn} is an urn containing balls of two different colors, say white and blue. We start with an urn of a certain number of balls. After each point of time, we draw a ball from the urn at random, observe its color and put it back to the urn. If a white (blue) ball is drawn, then we add $a$ ($c$) white balls and $b$ ($d$) blue ones to the urn. These dynamics of the urn scheme are governed by a {\em replacement matrix}, $\left(\begin{smallmatrix}
a & b 
\\ c & d
\end{smallmatrix}\right)$.
The entries in the replacement matrix do not have to take deterministic values. A randomized replacement matrix is given by $\left(\begin{smallmatrix} 
\mathcal{W} & \mathcal{X} 
\\ \mathcal{Y} & \mathcal{Z} \end{smallmatrix}\right)$, where $\mathcal{W}$,  $\mathcal{X}$, $\mathcal{Y}$, and $\mathcal{Z}$ are random variables with integer-value support.

An urn is said to be {\em balanced} if the total number of balls added is constant, i.e., $a + b = c + d = k$ or $\mathcal{W} + \mathcal{X} = \mathcal{Y} + \mathcal{Z} = k$. The row sum $k$ is called the {\em balance factor}. An urn is said to be {\em tenable} if we always can perpetuate the drawing according to replacement rules on every stochastic path, but never get ``stuck.''

A {\em \polya\ process} is a renewal process obtained by embedding a \polya\ urn scheme into continuous time. It was probably in~\cite{Athreya} that a \polya\ urn model (more precisely, the {\em Bernard Friedman's urn}~\cite{Friedman} was considered in a continuous-time setting. We pick the terminology ``\polya\ process'' from a recent research article~\cite{Sparks}.

The rest of the note is organized as follows: In Section~\ref{Sec:methodologies}, we define \polya\ process (in Section~\ref{Sec:Polya}) and elaborate the methodologies for this study (in Section~\ref{Sec:methods}). We then demonstrate the applications of our method explicitly via two examples: the Bagchi-Pal processes in Section~\ref{Sec:BagchiPal} and a class of randomized \polya\ processes in Section~\ref{Sec:randommatrix}. In particular, we look into the \polya\ process obtained by embedding the randomized Play-the-Winner scheme into continuous time. Some numeric studies are conducted in Section~\ref{Sec:simulation}. Finally, some concluding remarks and and possible directions of future research are given in Section~\ref{Sec:concluding}.

\section{\polya\ processes and methodologies}
\label{Sec:methodologies}
In this section, we formally define \polya\ processes, and then propose a method of bootstrapped moments to characterize the asymptotic behavior of \polya\ processes.

\subsection{\polya\ process}
\label{Sec:Polya}
A (two-color) \polya\ process is obtained by embedding a two-color \polya\ urn scheme into continuous time. The process evolves over time according to some specified rules. Initially, the urn contains a certain number of white and blue balls. Each ball in the urn is endowed with an independent clock that rings in time $\Exp(1)$. When the clock of a ball rings (at a renewal point), the ball is immediately picked from the urn, its color is observed; the ball is then instantaneously placed back in the urn, and the rule is executed. We do not count the time loss of the selection of the ball and the execution of the rules. All new balls are endowed with their own independent clocks. The process progresses in this manner. By the memoryless property of exponential interarrival times, the process is reset to start at every epoch. The \polya\ process is Markovian, and the rate of the process changes owing to the number of ball additions.

Let $W(t)$ and $B(t)$ be the number of white and blue balls in an urn at time $t \ge 0$, respectively. Our goal is to determine the limiting distribution of the process $\bigl(W(t), B(t)\bigr)^{\top}$ (after properly scaled) as $t \to \infty$. This research topic is classical and has been extensively studied in the \polya\ urn model community. The existence of the asymptotic distribution of $\bigl(W(t), B(t)\bigr)^{\top}$ was proved in the seminal paper~\cite{Janson2004}, and the distribution was first determined in~\cite[Theorem 3.1]{Janson2004} and revisited by~\cite[Theorem 3.2]{Chauvin}. 

The main contribution of this note is to propose an elementary but effective method for determining the asymptotic mixed moments of the process $\bigl(W(t), B(t)\bigr)^{\top}$, and consequently characterizing its limiting behavior. Our approach only requires fundamental knowledge of calculus and linear algebra, whereas the proof in~\cite[pages 216--230]{Janson2004} was based on a sophisticated martingale method, in coupled with other results of branching processes~\cite{Ney}. The method in this note provides an alternative technique which can be broadly applied to solving many similar stochastic process problems. Furthermore, our method appears to be a good option for the scientists and researchers who have interests in stochastic analysis but have limited knowledge in advanced probability theory.

\subsection{Methodologies}
\label{Sec:methods}
We present our main methods in this section. The section is divided into two subsections. In Subsection~\ref{Subsec:PDE}, we give two kinds of partial differential equations (PDEs) that respectively govern tenable \polya\ processes associated with deterministic replacement matrices and with randomized replacement matrices. In Subsection~\ref{Subsec:bootstrap}, we introduce a method in a bootstrapped manner to compute the mixed moments of the process. This method is extremely useful when the solutions to the PDEs established in Subsection~\ref{Subsec:PDE} is unwieldy.

\subsubsection{Partial differential equations}
\label{Subsec:PDE}
Consider the joint moment generating function of $\bigl(W(t), B(t)\bigr)^{\top}$: $\phi(t, u, v) := \E \left[e^{uW(t) + vB(t)}\right]$.
For tenable \polya\ urns (not necessary to be balanced) associated with $\left(\begin{smallmatrix}
a & b 
\\ c & d
\end{smallmatrix}\right)$, the joint process is governed by the PDE:
\begin{equation}
	\frac{\partial \phi(t,u,v)} {\partial t} + \bigl(1 - e^{au + bv}\bigr) \frac{\partial \phi(t,u,v)} {\partial u} + \bigl( 1 - e^{cu + dv}\bigr)\frac{\partial \phi(t,u,v)} {\partial v} = 0;
	\label{Eq:PDEgeneral} 
\end{equation}
see~\cite[Lemma 2.1]{Balaji}. The general solution to PDE~(\ref{Eq:PDEgeneral}) is known, but it is an integration along characteristic curves, where the characteristics are difficult to determine in most cases. To the best of our knowledge, there is a very limited number of cases such that the solution to PDE~(\ref{Eq:PDEgeneral}) is given by a closed form, as for instance the case of forward and backward diagonal processes~\cite{Balaji}, the Ehrenfest processes~\cite{Watanabe}, and zero-balanced processes with replacement matrix of Bernoulli entries~\cite{Sparks}. 

Analogously, there is a PDE that governs the joint process $\bigl(W(t), B(t)\bigr)^{\top}$ associated with $\left(\begin{smallmatrix}
\mathcal{W} & \mathcal{X} 
\\ \mathcal{Y} & \mathcal{Z}
\end{smallmatrix}\right)$, developed in~\cite[Lemma 4.1]{Mahmoud}; that is,
\begin{equation}
	\frac{\partial \phi} {\partial t} + \bigl(1 - \psi_{\mathcal{W},\mathcal{X}}(u,v)\bigr) \frac{\partial \phi} {\partial u} + \bigl( 1 - \psi_{\mathcal{Y},\mathcal{Z}}(u,v)\bigr)\frac{\partial \phi} {\partial v} = 0,
	\label{Eq:PDErandom} 
\end{equation}
where $\psi_{\mathcal{W},\mathcal{X}}$ and $\psi_{\mathcal{Y},\mathcal{Z}}$ are the joint moment generating functions of $\mathcal{W}$ and $\mathcal{X}$ and $\mathcal{Y}$ and $\mathcal{Z}$, respectively.

In the next subsection, we shall develop a method to extract mixed moments of \polya\ process, in a bootstrapped way, when the analytical solution to functional equations~(\ref{Eq:PDEgeneral}) or~(\ref{Eq:PDErandom}) is intractable or not in a closed form.

\subsubsection{Method of bootstrapped moments}
\label{Subsec:bootstrap}
In this subsection, we introduce a method to calculate all mixed moments of \polya\ process. We demonstrate our method only for PDE~(\ref{Eq:PDEgeneral}), and that for PDE~(\ref{Eq:PDErandom}) can be done in a similar manner.

We consider the mixed moments of $W(t)$ and $B(t)$, i.e., $\E\bigl[W^{i}(t) B^{j}(t)\bigr]$ for $i \ge 0$, $j \ge 0$, and $i + j \neq 0$. For $t \ge 0$, the mixed moment $\E\bigl[W^{i}(t) B^{j}(t)\bigr]$ can be obtained by applying the differential operator, $\frac {\partial^{i+j}} {\partial u^i \, \partial v^j}$, to the joint moment generating function $\phi(t, u, v)$, and then evaluating it for $u = v = 0$; that is,
$$ \frac {\partial^{i+j}} {\partial u^i \, \partial v^j}  \phi(t,u,v) \Big|_{u=v=0}  = \E\left[\frac {\partial^{i+j}} {\partial u^i \,\partial v^j} e^{W(t) u + B(t) v}\Big|_{u=v=0}\right] = \E\bigl[W^i(t) B^j(t)\bigr].$$

Apply the differential operator $\frac {\partial^{i+j}} {\partial u^i \partial v^j}$ to PDE~(\ref{Eq:PDEgeneral}), and set $u = v = 0$, i.e.,
\begin{align}
	&\frac {\partial^{i+j}} {\partial u^i \, \partial v^j} \Bigl(\frac {\partial} {\partial t} \phi(t,u,v)\Bigr) \Big|_{u=v=0} + \frac {\partial^{i+j}} {\partial u^i \, \partial v^j} \bigl(1 - e^{au+bv}\bigr) \frac {\partial \phi} {\partial u } \Big|_{u=v=0}
	\nonumber \\
	&\qquad \qquad + \frac {\partial^{i+j}} {\partial u^i\, \partial v^j} 
	\bigl(1 - e^{cu + dv}\bigr) \frac {\partial \phi} {\partial v } \Big|_{u=v=0} = 0. 
	\label{Eq:differential}
\end{align}

We evaluate the three terms in the left-hand side of Equation~(\ref{Eq:differential}) one after another, and thus obtain ordinary differential equations (ODEs) for mixed moments of $W(t)$ and $B(t)$:
\begin{align}
	\frac d {d t}  \E\bigl[W^i(t) B^j(t)\bigr] 
	&= \sum_{r=0}^{i-1} {i \choose r} a^{i-r} \E\bigl[W^{r+1} (t)B^{j}(t)\bigr] 
	\nonumber \\
	& \quad + \sum_{r=0}^i \sum_{s=0}^{j-1}  
	{i \choose r}  {j \choose s} a^{i-r} b^{j-s}   
	\E\bigl[W^{r+1} (t)B^{s}(t)\bigr]
	\nonumber \\ 
	& \quad + \sum_{r=0}^{i-1} {i \choose r} c^{i-r} \E\bigl[W^{r} (t)B^{j+1}(t)\bigr] 
	\nonumber \\
	& \quad +  \sum_{r=0}^i  \sum_{s=0}^{j-1}
	{i \choose r}  {j \choose s} c^{i-r} d^{j-s}   
	\E\bigl[W^{r} (t)B^{s+1}(t)\bigr].
	\label{Eq:ODEgeneral}
\end{align}

Let $i + j$ be the {\em order} of mixed moment $\E\bigl[W^i(t) B^j(t)\bigr]$. Equation~(\ref{Eq:ODEgeneral}) suggests that, for any fixed $i + j$, $\E\bigl[W^i(t) B^j(t)\bigr]$ can be computed if we know all the lower-order mixed moments (i.e., $\E\bigl[W^r(t) B^s(t)\bigr]$ for $1 \le r + s < i + j$) and some mixed moments of the same order (i.e., $\E\bigl[W^r(t) B^s(t)\bigr]$ for $r + s = i + j$). We propose a strategy to extract all mixed moments in a bootstrapped manner (i.e., following a particular order) from the ODEs.  We first calculate the mixed moments of the lowest order, i.e., $\E\bigl[W(t)\bigr]$ and $\E\bigl[B(t)\bigr]$. When computing the mixed moments of order $i + j \ge 2$, we plug the solutions of all the mixed moments of order less than $i + j$ into $i + j + 1$ ODEs, and solve all mixed moments of order $i + j$ simultaneously. It is expected that the ODEs of the mixed moments of high order are complicated, since the establishment of those equations requires all mixed moments of lower order. In addition, the higher the order is, the more equations we need to solve. Although the exact solutions of mixed moments are too complex to get, the leading terms are attainable, and we can use them to characterize the asymptotic behavior of the process. Two classes of \polya\ processes are used as examples to demonstrate our methods, respectively presented in Sections~\ref{Sec:BagchiPal} and~\ref{Sec:randommatrix}.

\section{Application to the Bagchi-Pal processes}
\label{Sec:BagchiPal}
In this section, we investigate the {\em Bagchi-Pal processes}, a class of \polya\ processes obtained by embedding the {\em Bagchi-Pal urn} schemes into continuous time. The Bagchi-Pal urn is a generalized \polya-Eggenberger urn~\cite{Egg}, originally used for estimating the computer memory requirements of {\em 2-3 trees}~\cite{Bagchi}. The replacement matrix associated to the Bagchi-Pal urn is given by $\left(\begin{smallmatrix}
a & b
\\ c & d
\end{smallmatrix} \right) = \left(\begin{smallmatrix}
k - b & b
\\ c & k - c
\end{smallmatrix} \right)$. The Bagchi-Pal urn is balanced and tenable as well; see~\cite{Bagchi} for the discussion of tenability. For $t \ge 0$, denote $\tau(t) = W(t) + B(t)$ the total number of balls in the urn. Consider a Bagchi-Pal urn starting with $W_0$ white balls and $B_0$ balls, respectively. 

\begin{theorem}
	\label{Thm:limitdist}
	The asymptotic mixed joint moments are
	$$\lim_{t \to \infty} \frac{\E \bigl[W^i(t) B^j(t)\bigr]}{e^{k(i + j)t}} = \frac{b^j c^i}{(b + c)^{i + j}} k^{i + j} \left\langle \frac{\tau_0}{k} \right\rangle_{i + j},$$
	where $\tau_0 = W_0 + B_0$. Accordingly, we get
	\begin{equation*}
		e^{-kt} \begin{pmatrix} W(t) \\ B(t) \end{pmatrix} \convD {\rm Gamma} \left(\frac{\tau_0}{k}, k\right) \begin{pmatrix} \frac{c}{b+c} \\ \frac{b}{b+c} \end{pmatrix}.
	\end{equation*}
\end{theorem}

\begin{proof}
	The PDE that governs the Bagchi-Pal process is 
	\begin{equation}
		\frac{\partial \phi} {\partial t} + \bigl(1 - e^{(k-b)u + bv}\bigr) \frac{\partial \phi} {\partial u} + \bigl( 1 - e^{cu+(k-c)v}\bigr)\frac{\partial \phi} {\partial v} = 0.
		\label{Eq:Bagchi-Pal}
	\end{equation}
	The analytical solution to Equation~{\ref{Eq:Bagchi-Pal}} is unwieldy. We resort to the ODE of $\E\bigl[W^i(t) B^j(t)\bigr]$: 
	\begin{align}
		\frac d {d t}  \E\bigl[W^i(t) B^j(t)\bigr] 
		&= \sum_{r=0}^{i-1} {i \choose r} (k-b)^{i-r} \E\bigl[W^{r+1} (t)B^{j}(t)\bigr] 
		\nonumber \\
		& \quad + \sum_{r=0}^i \sum_{s=0}^{j-1}  
		{i \choose r}  {j \choose s} (k-b)^{i-r} b^{j-s}   
		\E\bigl[W^{r+1} (t)B^{s}(t)\bigr]
		\nonumber \\ 
		& \quad + \sum_{r=0}^{i-1} {i \choose r} c^{i-r} \E\bigl[W^{r} (t)B^{j+1}(t)\bigr] 
		\nonumber \\
		& \quad +  \sum_{r=0}^i  \sum_{s=0}^{j-1}
		{i \choose r}  {j \choose s} c^{i-r} (k-c)^{j-s}   
		\E\bigl[W^{r} (t)B^{s+1}(t)\bigr].
		\label{Eq:mixed}
	\end{align}
	
	To simplify the notation, we denote $ m_{i,j} (t)= \E\bigl[W^i(t)B^j(t)\bigr]$. In the next lemma, we prove that the leading term in $m_{i,j} (t)$ is an exponential function of power $k(i + j)t$.
	\begin{lemma}
		\label{Lem:asymom}
		For $i, j \ge 0$ and $i + j \ge 1$, we have
		\begin{equation}
			m_{i, j}(t)  =  K_{i, j} e^{k(i+j)t} + O\left(e^{[k(i+j)-(b+c)]t}\right) + O\left(e^{k(i+j-1)t}\right),
			\label{Eq:exp} 
		\end{equation} 
		where $K_{i, j} =  K_{i, j}(b,c,k,W_0,B_0) \in \mathbb R$ are the coefficients for the leading terms.
	\end{lemma}
	
	\begin{proof}
		We prove the lemma by an induction on order of mixed moments, i.e., $i + j$. Rewrite Equation~(\ref{Eq:mixed}) in terms of $m_{i, j}(t)$ and separate the mixed moments of order $i + j$ from those of lower order; that is,
		\begingroup
		\allowdisplaybreaks
		\begin{align}
			\frac{d}{dt} m_{i,j} (t)  
			&=  \bigl[jb m_{i+1,j-1} (t)\bigr] + \bigl[\bigl( i(k-b) + j(k-c) \bigr) m_{i,j} (t)\bigr] + \bigl[ic m_{i-1,j+1} (t)\bigr] 
			\nonumber \\ 
			&\qquad +  \sum_{r=0}^{i-2} {i \choose r} (k-b)^{i-r}  m_{r+1,j} (t) + \sum_{r=0}^{i-2} {i \choose r} c^{i-r} m_{r,j+1}(t)
			\nonumber \\
			& \qquad + \sum_{r=0}^{i-1} \sum_{s=0}^{j-1} {i \choose r}  {j \choose s} (k-b)^{i-r} b^{j-s} m_{r+1,s} (t) 
			\nonumber \\ 
			& \qquad {} +  \sum_{r=0}^{i-1} \sum_{s=0}^{j-1} {i \choose r}  {j \choose s} c^{i-r} (k-c)^{j-s}  m_{r,s+1} (t) 
			\nonumber \\
			& \qquad + \sum_{s=0}^{j-2} {j \choose s} b^{j-s} m_{i+1,s} (t)
			+ \sum_{s=0}^{j-2} {j \choose s} (k-c)^{j-s}  m_{i,s+1} (t).
			\label{Eq:first-diff}
		\end{align}
		\endgroup
		
		Denote $i + j = n$. The base of the induction is $n = 1$, which is either $\{i = 1, j = 0\}$ or $\{i = 0, j = 1\}$. The corresponding differential equations are
		\begin{align*}
			\frac{d}{d t} \, \E \bigl[W(t)\bigr] &= (k-b)\, \E \bigl[W(t)\bigr] + c\, \E \bigl[B(t)\bigr],\\
			\frac{d}{d t} \, \E \bigl[B(t)\bigr] &= b\, \E \bigl[W(t)\bigr] + (k-c)\, \E \bigl[B(t)\bigr],
		\end{align*}
		which jointly form a standard differential equation system with solutions
		\begin{align*}
			\E\bigl[W(t)\bigr] &= \frac{c}{b+c}\tau_0 e^{kt} + \frac{cB_0-bW_0}{b+c}e^{(k-b-c)t},\\
			\E\bigl[B(t)\bigr] &= \frac{b}{b+c}\tau_0 e^{kt} - \frac{cB_0-bW_0}{b+c}e^{(k-b-c)t}.
		\end{align*}
		The base is verified. Let $\matM_n(t)$ be an $(n + 1) \times 1$ vector that contains all mixed moments of order $n$, i.e., $\matM_n(t) = \left( m_{n,0}(t), \ldots, m_{1,n-1}(t),
		m_{0,n}(t)\right)^{\top}$ and $\matH_n(t)$ be an $\bigl((n^2 + n - 2)/2\bigr) \times 1$ vector that contains all mixed moments of order up to $n - 1$, i.e., $\matH_n(t) = \left(m_{1,0}(t), m_{0,1}(t), \ldots, m_{0,n-1}(t)\right)^{\top}$. Thus, the ODEs (c.f.\ Equation~(\ref{Eq:first-diff})) can be represented by the following non-homogeneous matrix differentiation equation:                   
		\begin{eqnarray}
			\frac d {d t} \matM_n(t) = \matA_n \matM_n(t) + \matB_n \matH_{n}(t),
			\label{Eq:matrix}
		\end{eqnarray} 
		where $\matB_n$ is an $(n + 1) \times \bigl((n^2 + n - 2)/2\bigr)$ matrix free of $t$, and $\matA_n$ is an $(n + 1) \times (n + 1)$ tridiagonal matrix, which can be expressed explicitly as follows:
		$$\matA_n =  \begin{pmatrix} n\alpha & nc & 0 & \cdots & 0 & 0\\
		b & n\alpha-\beta & (n-1)c & \cdots & 0 & 0\\
		\vdots & \vdots & \vdots & \ddots & \vdots & \vdots \\
		0 & 0 & 0 & \cdots & n\alpha-(n-1)\beta & c\\
		0 & 0 & 0 & \cdots & nb & n\alpha- n\beta\end{pmatrix}$$
		with $\alpha = k-b$ and $\beta = c-b$. The tridiagonal matrix $\matA_n$ is a member of {\em Leonard pairs of the Krawtchouk type}; see~\cite{Terwilliger} for details. This type of matrix is known to have real eigenvalues $\lambda_0 > \lambda_1 > \cdots > \lambda_{n}$ forming an arithmetic progression. Based on this distinct feature, we are able to compute all eigenvalues for $\matA_n$. They are $\lambda_s = nk - s(b+c)$ for $s = 0,1,\ldots,n$.
		
		We are now ready to prove the inductive step. Assume that Equation~(\ref{Eq:exp}) holds for all mixed moments with order $i + j$ up to $n - 1$. For $i + j = n$, the vector $\matH_n(t)$ in Equation~(\ref{Eq:matrix}) can be represented as
		$$\matH_n(t) =  \begin{pmatrix}  K_{1, 0} \, e^{kt} + O(e^{[k-(b+c)]t}) + O(1) \\
		K_{0, 1} \, e^{kt} + O(e^{[k-(b+c)]t}) + O(1)  \\
		\vdots \\
		K_{0,n-1} \, e^{(n-1)kt} + O(e^{[(n-1)k-(b+c)]t}) + O(e^{(n-2)kt}) \end{pmatrix}.$$
		Noticing that the coefficient matrix $\matB_n$ is free of $t$, we thus conclude 
		$$\matB_n\matH_n(t) = \left(O(e^{(n-1)kt}), O(e^{(n-1)kt}), \ldots, O(e^{(n-1)kt})\right)^{\top}.$$
		
		Note that Equation~(\ref{Eq:matrix}) is in a class of non-homogeneous matrix differential equations known to have a general solution~\cite{Edwards}; that is,
		\begin{equation}
			\matM_n(t) = e^{\matA_n t} \matM_n(0) + \int_{0}^{t} e^{-\matA_n(x-t)}\matB_n\matH_n(x) \, dx.
			\label{Eq:matrixsolution}
		\end{equation} 
		We evaluate the two terms in the right-hand side of Equation~(\ref{Eq:matrixsolution}) one by one. We apply the {\em Sylvester's formula} to the first term to get
		\begin{equation}
			e^{\matA_n t} \matM_n(0) = \left(\sum_{s=0}^{n} e^{\lambda_s t} {\cal E}_s \right)\matM_n(0) =  \begin{pmatrix}   \xi_{n,0} \, e^{nk t} + O(e^{[nk -(b+c)]t})\\
				\xi_{n-1,1} \, e^{nk t} + O(e^{[nk -(b+c)]t})\\
				\vdots \\
				\xi_{0,n} \, e^{nk t} + O(e^{[nk -(b+c)]t})\end{pmatrix},  
			\label{Eq:matrix1}    
		\end{equation}
		where ${\cal E}_s$ are idempotent matrices and $\xi_{i,n-i} = \xi_{i, n - i}(b,c,k,W_0,B_0) \in \mathbb R$, for $i = 0, 1, \ldots, n$, are the coefficients for the leading terms therein. We compute the second term in a similar manner, and obtain
		\begingroup
		\allowdisplaybreaks
		\begin{equation*}
			\int_{0}^{t} e^{-\matA_n(x-t)} \matB_n \matH_n(x) dx = \begin{pmatrix}   \pi_{n, 0} \, e^{nkt} + O\bigl(e^{[nk -(b+c)]t}\bigr) + O\bigl(e^{(n-1)kt}\bigr) \\
				\pi_{n-1,1} \, e^{nkt} + O\bigl(e^{[nk -(b+c)]t}\bigr) + O\bigl(e^{(n-1)kt}\bigr)\\
				\vdots \\
				\pi_{0,n} \, e^{nkt} + O\bigl(e^{[nk -(b+c)]t}\bigr) + O\bigl(e^{(n-1)kt}\bigr)\end{pmatrix},
		\end{equation*}
		\endgroup
		where $\pi_{i,n-i} = \pi_{i,n-i}(b,c,k,W_0,B_0) \in \mathbb R$ are the coefficients for the leading terms.
		The proof is completed by putting two results together.
	\end{proof}
	
	The derivation of the asymptotic mixed moments for $W(t)$ and $B(t)$ needs the moments for the total number of balls, $\tau(t)$. The moments for $\tau(t)$ are known. The moment generating function of $\tau(t)$ was determined in~\cite{Balaji}, and the exact moments for $\tau(t)$ were given in terms of Stirling numbers of the second kind; see~\cite[Section 5]{Chen}. We state the exact moments for $\tau(t)$ without proof in the next lemma. We would like to remark that the result is not only true for the Bagchi-Pal processes, but also all \polya\ processes associated with tenable and balanced \polya\ urn schemes. 
	\begin{lemma}
		\label{Lm:total}
		For $n \ge 1$, the moments of $\tau(t)$ are
		\begin{equation}
			\E \bigl[\tau^n(t)\bigr] = k^n \sum_{i=1}^{n} (-1)^{n-i} \stirling{n}{i} \left\langle \frac{\tau_0}{k} \right\rangle_i \, e^{kit}.
			\label{Eq:moments of total}
		\end{equation}
	\end{lemma}
	
	The last task is to calculate the coefficients $K_{i, j} = K_{i, n - i}$, for $i = 0, 1, \ldots, n$.
	
	We write the mixed moments $m_{i, n - i}(t)$ in Equation~(\ref{Eq:first-diff}) in terms of those given in Equation~(\ref{Eq:exp}), and obtain a recurrence for $K_{i, n - i}$:
	$$
	(nc + ib - ic)K_{i, n - i} = (n - i)bK_{i + 1, n - i - 1} + ic K_{i - 1, n - i + 1},
	$$
	with the initial condition $K_{1, n - 1} = (c/b) K_{0, n}$. The solution is given by $K_{i, n - i} = (c/b)^i K_{0, n}.$ Noticing that $\E\bigl[\tau^n(t) \bigr] = \E \bigl[(W(t) + B(t))^n \bigr]$, we apply the {\em Binomial Theorem} and obtain
	$$ \sum_{i=1}^{n} k^n (-1)^{n-i} \stirling{n}{i} \left\langle \frac{\tau_0}{k} \right\rangle_i \, e^{kit} =  \E \bigl[\tau^n(t)\bigr] = \sum_{i=0}^{n} {n \choose i} m_{i, n - i}(t).$$
	Dividing both sides of the last display by $e^{knt}$, and letting $t$ go to infinity, we have
	$$ \lim_{t \to \infty}  \sum_{i=1}^{n} k^n (-1)^{n-i} \stirling{n}{i} \left\langle \frac{\tau_0}{k} \right\rangle_i \, \frac{e^{kit}}{e^{knt}} = \lim_{t \to \infty} \sum_{i=0}^{n} {n \choose i} \frac{m_{i, n - i}(t)}{e^{knt}},$$
	leading to a simple linear equation for $K_{0, n}$:
	$$k^n \left\langle \frac{\tau_0}{k} \right\rangle_n = K_{0,n} \left(1 + \frac{c}{b}\right)^n.$$
	We thus have
	$$K_{0,n}  = \left(\frac{b}{b + c}\right)^n k^n \left\langle \frac{\tau_0}{k} \right\rangle_n \qquad \mbox{and} \qquad K_{i, n - i} = \frac{b^{n - i} c^i}{(b + c)^n} k^{n} \left\langle \frac{\tau_0}{k} \right\rangle_n.$$
	Recalling Lemma~\ref{Lem:asymom}, we replace $n - i$ by $j$ in the last display to get
	$$
	\lim_{t \to \infty} \frac{m_{i, j}(t)}{e^{k(i + j)t}} = K_{i, j} = \frac{b^j c^i}{(b + c)^{i + j}} k^{i + j} \left\langle \frac{\tau_0}{k} \right\rangle_{i + j},
	$$
	which immediately leads to the results stated in the theorem.
\end{proof}

Before closing this section, we give two remarks. First, the asymptotic Pearson's correlation coefficient between $W(t)/e^{kt}$ and $B(t)/e^{kt}$ is equal to $1$, which is an instantaneous corollary of Theorem~\ref{Thm:limitdist}. Second, asymptotic mixed moments of $W(t)$ and $B(t)$ only depend on the initial total number of balls ($\tau_0$) in the urn, but not specifically on the initial number of white balls ($W_0$), nor on the initial number of blue balls ($B_0$). 

\section{Application to a class of randomized \polya\ processes}
\label{Sec:randommatrix}
In this section, we present an application of our method to a class of tenable and balanced processes with replacement matrix of random entries. Similar results (c.f.\ Theorem~\ref{Thm:limitdist}) are obtained for this class of processes with minor modifications in the proofs, done mutatis mutandis. Therefore, we will only state the major results, but omit those analogous arguments.

Let $k$ be the balance factor. The replacement matrix is $\left(\begin{smallmatrix}
\mathcal{W} & k - \mathcal{W}
\\ k - \mathcal{Z} & \mathcal{Z}
\end{smallmatrix}\right)$, where $\mathcal{W}$ and $\mathcal{Z}$ are discrete random variables. To avoid issues with tenability, let the support for $\mathcal{W}$, as well as $\mathcal{Z}$, be $\{0, 1, \ldots,k\}$. By the definition of joint moment generating function, we have
\begin{align*}
	\psi_{\mathcal{W}, k-\mathcal{W}}(u, v) &= \E\left[e^{\mathcal{W}u+(k-\mathcal{W})v}\right] = \sum_{w=0}^{k} \Prob(\mathcal{W} = w) e^{w u + (k-w) v}
	\\ \psi_{k-\mathcal{Z}, \mathcal{Z}}(u, v) &=  \E\left[e^{(k - \mathcal{Z})u+\mathcal{Z})v}\right] =\sum_{z=0}^{k} \Prob(\mathcal{Z}=z) e^{(k-z) u+ z v}
\end{align*}
Plugging the joint generating functions into Equation~(\ref{Eq:PDErandom}) and applying the differential operator to the equation, we obtain the ODEs for mixed moments and ultimately get a similar result as Lemma~\ref{Lem:asymom}; that is,
\begin{equation*}
	\E\bigl[W^i(t) B^j(t)\bigr]  =  M_{i, j} e^{k(i+j)t} + O\left(e^{[k(i+j)-(\mu_\mathcal{W}+\mu_\mathcal{Z})]t}\right) + O\left(e^{k(i+j-1)t}\right),
\end{equation*} 
where $M_{i, j} = \frac{(k - \mu_\mathcal{W})^{n - i}(k - \mu_\mathcal{Z})^{i}}{\left((k - \mu_\mathcal{W}) + (k - \mu_\mathcal{Z}) \right)^{i + j}} k^{i + j} \left\langle\frac{\tau_0}{k}\right\rangle_{i + j}$, and $\mu_\mathcal{W}$ and $\mu_\mathcal{Z}$ are the means of $\mathcal{W}$ and $\mathcal{Z}$, respectively. Accordingly, we have
\begin{equation}
\label{Eq:limitrandom}
	e^{-kt} \begin{pmatrix} W(t) \\ B(t) \end{pmatrix} \convD {\rm Gamma} \left(\frac{\tau_0}{k}, k\right) \begin{pmatrix} \frac{k - \mu_\mathcal{Z}}{2k - \mu_\mathcal{W} - \mu_\mathcal{Z}} \\ \frac{k - \mu_\mathcal{W}}{2k - \mu_\mathcal{W} - \mu_\mathcal{Z}} \end{pmatrix}.
\end{equation}

\polya\ processes with randomized replacement matrix have found applications in many fields. One of the most well-known examples is the randomized {\em Play-the-Winner} scheme. The randomized Play-the-Winner scheme is an adaptive design in clinical trials, proposed by~\cite{Wei}. It was probably first noted in~\cite{Wei1979} that the randomized Play-the-Winner scheme could be formulated by \polya\ urn models.

Consider the following senario. Suppose that there are two treatments, $T_1$ and $T_2$, and a clinician selects a treatment for the next patient according to the following rules. The clinician randomly selects a ball from a two-color (white and blue) urn, observes its color, and returns it back to the urn. If the ball is white, $T_1$ is given to the patient. If $T_1$ succeeds, one white ball is added to the urn; otherwise, one blue ball is added to the urn. On the other hand, if the ball drawn by the clinician is blue, $T_2$ is given to the patient. If $T_2$ succeeds, one blue ball is added to the urn; otherwise, one white ball is added to the urn. Such scheme is always in favor of the successful treatment. Suppose that the success rate for each of the treatments stays unchanged, say $p_1$ for $T_1$ and $p_2$ for $T_2$, the dynamics of such urn scheme can be represented by the replacement matrix $\left(\begin{smallmatrix}
\mathcal{B}_1 & 1 - \mathcal{B}_1
\\ 1 - \mathcal{B}_2 & \mathcal{B}_2
\end{smallmatrix}\right)$,
where $\mathcal{B}_1$ and $\mathcal{B}_2$ are Bernoulli random variables with success rates $p_1$ and $p_2$, respectively. According to the result in Equation~(\ref{Eq:limitrandom}), we obtain the following asymptotic distribution for the Poissonized \polya\ urns associated with the randomized Play-the-Winner scheme; namely,
\begin{equation*}
	e^{-t} \begin{pmatrix} W(t) \\ B(t) \end{pmatrix} \convD {\rm Gamma} (\tau_0, 1) \begin{pmatrix} \frac{q_2}{q_1 + q_2} \\ \frac{q_1}{q_1 + q_2} \end{pmatrix},
\end{equation*}
where $q_1 = 1 - p_1$ and $q_2 = 1 - p_2$. This yields that the proportion of patients assigned to $T_1$ converges to $q_2/(q_1 + q_2)$, which is consistent with the asymptotic result of the randomized Play-the-Winner scheme that progresses in discrete time; see~\cite{Rosenberger}.

\section{Simulation results}
\label{Sec:simulation}
We conduct some simulation studies in this section to evaluate and verify the theoretical results developed in Sections~\ref{Sec:BagchiPal} and~\ref{Sec:randommatrix}. 

We first consider the Bagchi-Pal process. Given $W_0$, $B_0$ (the initial number of white and blue balls in the urn), $t^{*}$ (a threshold of time that terminates the simulation), and $a$, $b$, $c$, $d$ (the four entries in the ball addition replacement matrix) such that $k = a + b = c + d$, our Monte-Carlo experiment proceeds as follows:

For each Monte-Carlo replica $m$, we start by generating $\tau = W + B$ (initially $W = W_0$ and $B = B_0$) independent $\Exp(1)$ random variables as clocks for each of the balls in the urn at $t_0 = 0$. At time point $t = t_0 + \Exp(1/\tau)$, a clock rings. To implement the ball addition rule, we generate an independent random variable $U \sim {\rm Unif}(0, 1)$. If $U \le W/(W + B)$, we update $W = W + a$ and $B = B + b$; otherwise, we update $W = W + c$ and $B = B + d$. Lastly, we update $\tau = \tau + k$, and restart the iteration at the next renewal point. We continue iterations in this manner until time $t$ exceeds the given threshold~$t^{*}$. We evaluate the proportion of white balls in the urn for each replica $m = 1, 2, \ldots, M$, and graphically compare the estimates with our theoretical results.
\begin{figure}[ht]
	\centering
	\begin{minipage}[b]{0.45\textwidth}
		\includegraphics[width=\textwidth]{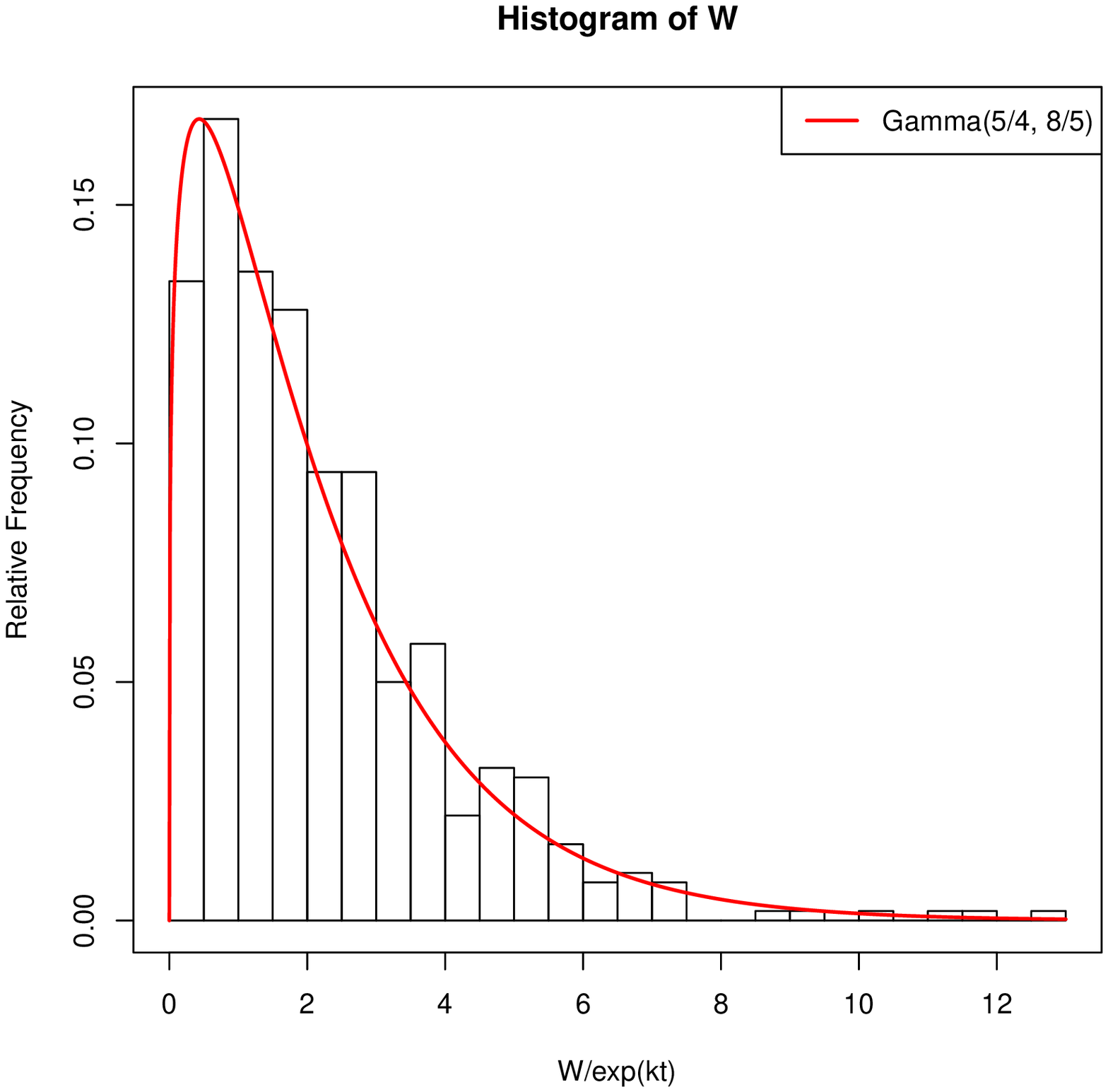}
		\caption{Histogram of $W(t)/e^{kt}$ ($k = 4$) of the Bagchi-Pal process.}
		\label{Fig:BPw}
	\end{minipage}
	\hfill
	\begin{minipage}[b]{0.45\textwidth}
		\includegraphics[width=\textwidth]{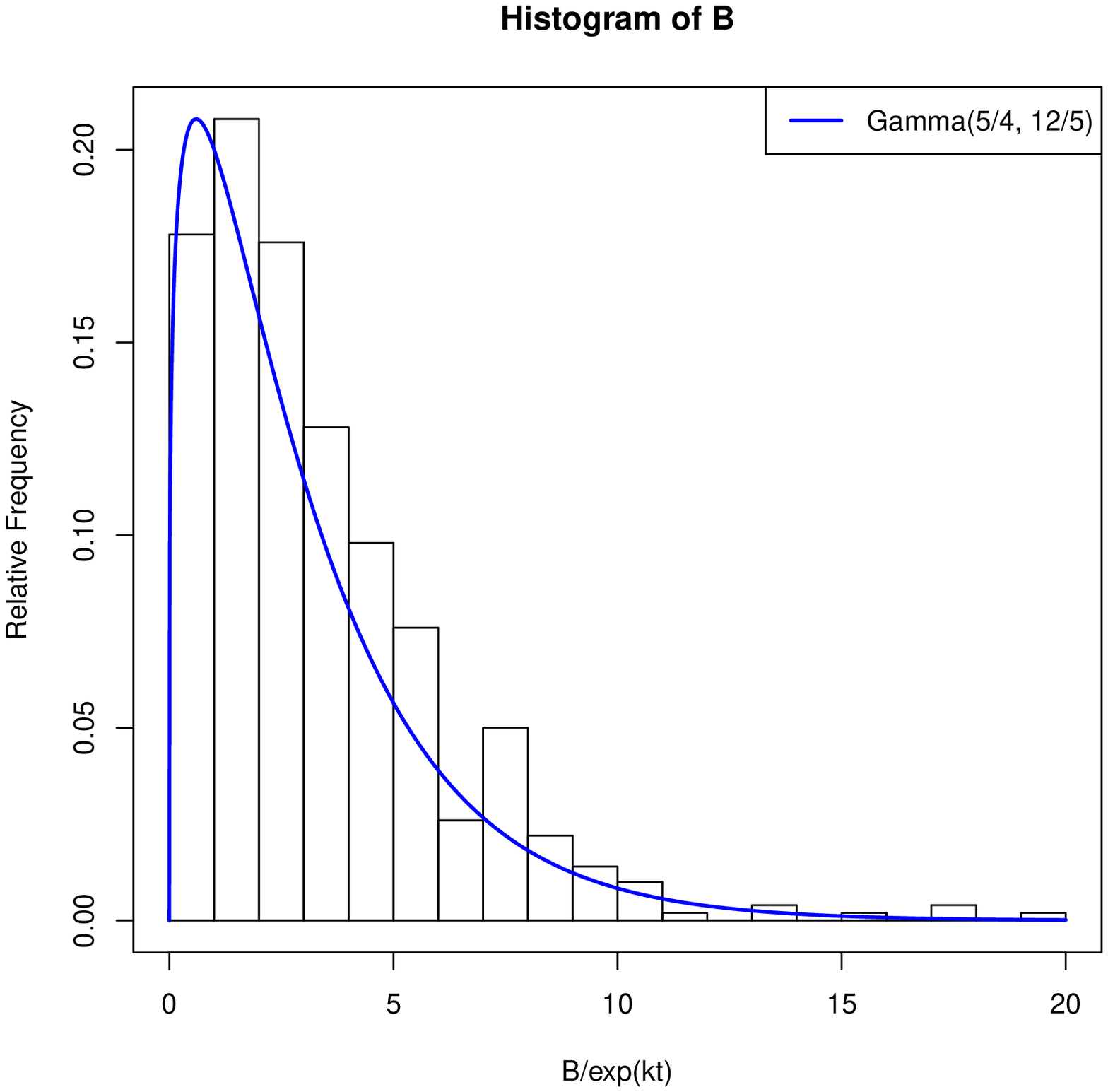}
		\caption{Histogram of $B(t)/e^{kt}$ ($k = 4$) of the Bagchi-Pal process.}
		\label{Fig:BPb}
	\end{minipage}
\end{figure}

Figures~\ref{Fig:BPw} and~\ref{Fig:BPb} depict illustrative results from our Monte-Carlo experiment, with $W_0 = 3$, $B_0 = 2$, and the replacement matrix is $\left(\begin{smallmatrix} 
1 & 3 
\\ 2 & 2 \end{smallmatrix}\right)$. The threshold $t^*$ is set at $2$, where the ball addition rule is executed about $4112$ times (in average). The number of simulations $M$ is $500$. The histograms of $W(t)$ and $B(t)$ after properly scaled ($e^{4t}$) are compared with the probability density functions of their associated limiting distributions (respectively ${\rm Gamma}(5/4, 8/5)$ and ${\rm Gamma}(5/4, 12/5)$). 

We compute the sample proportion of the number of white balls in the urn over $500$ replicates; that is $0.39998$, cf.\ the theoretically asymptotic proportion $c / (b + c) = 2/5$. In addition, the sample correlation between $W(t)/e^{4t}$ and $B(t)/e^{4t}$ is $0.99997$.

We conduct an analogous numerical study for \polya\ processes with randomized replacement matrix. We take the randomized Play-the-Winner scheme as an example. We again set initial conditions at $W_0 = 3$ and $B_0 = 2$, and select $\mathcal{B}_1 \sim {\rm Bernoulli}(p_1)$ and $\mathcal{B}_2 \sim {\rm Bernoulli}(p_2)$ for $p_1 = 3/10$ and $p_2 = 6/10$. The threshold $t^*$ is set at $7$, and $500$ replications are simulated. The histograms of $W(t)$ and $B(t)$ after properly scaled ($e^t$) are depicted in Figures~\ref{Fig:PWw} and~\ref{Fig:PWb}.

\begin{figure}[ht]
	\centering
	\begin{minipage}[b]{0.45\textwidth}
		\includegraphics[width=\textwidth]{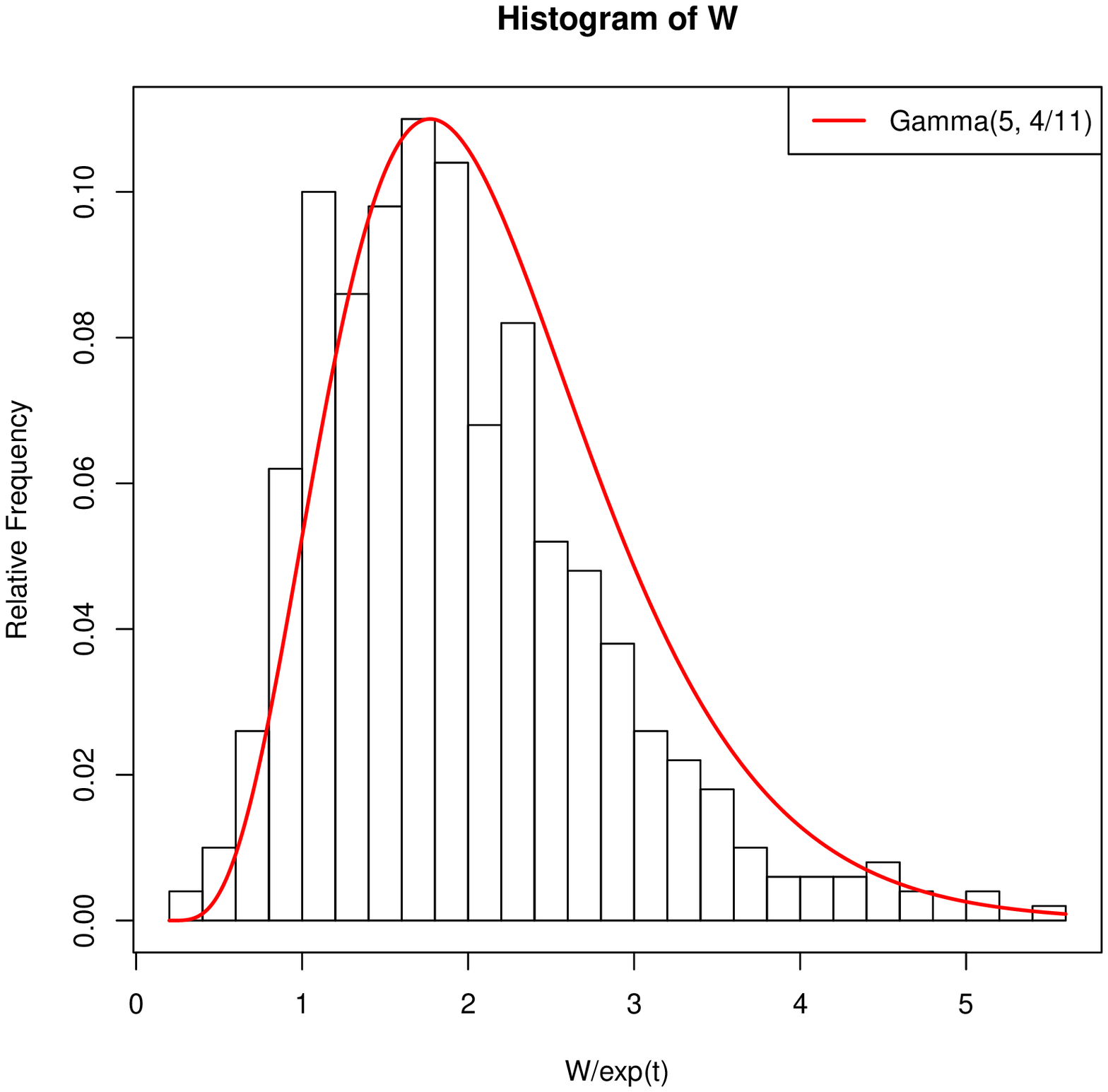}
		\caption{Histogram of $W(t)/e^{t}$ of the P\'{o}lya process with randomized replacement matrix.}
		\label{Fig:PWw}
	\end{minipage}
	\hfill
	\begin{minipage}[b]{0.45\textwidth}
		\includegraphics[width=\textwidth]{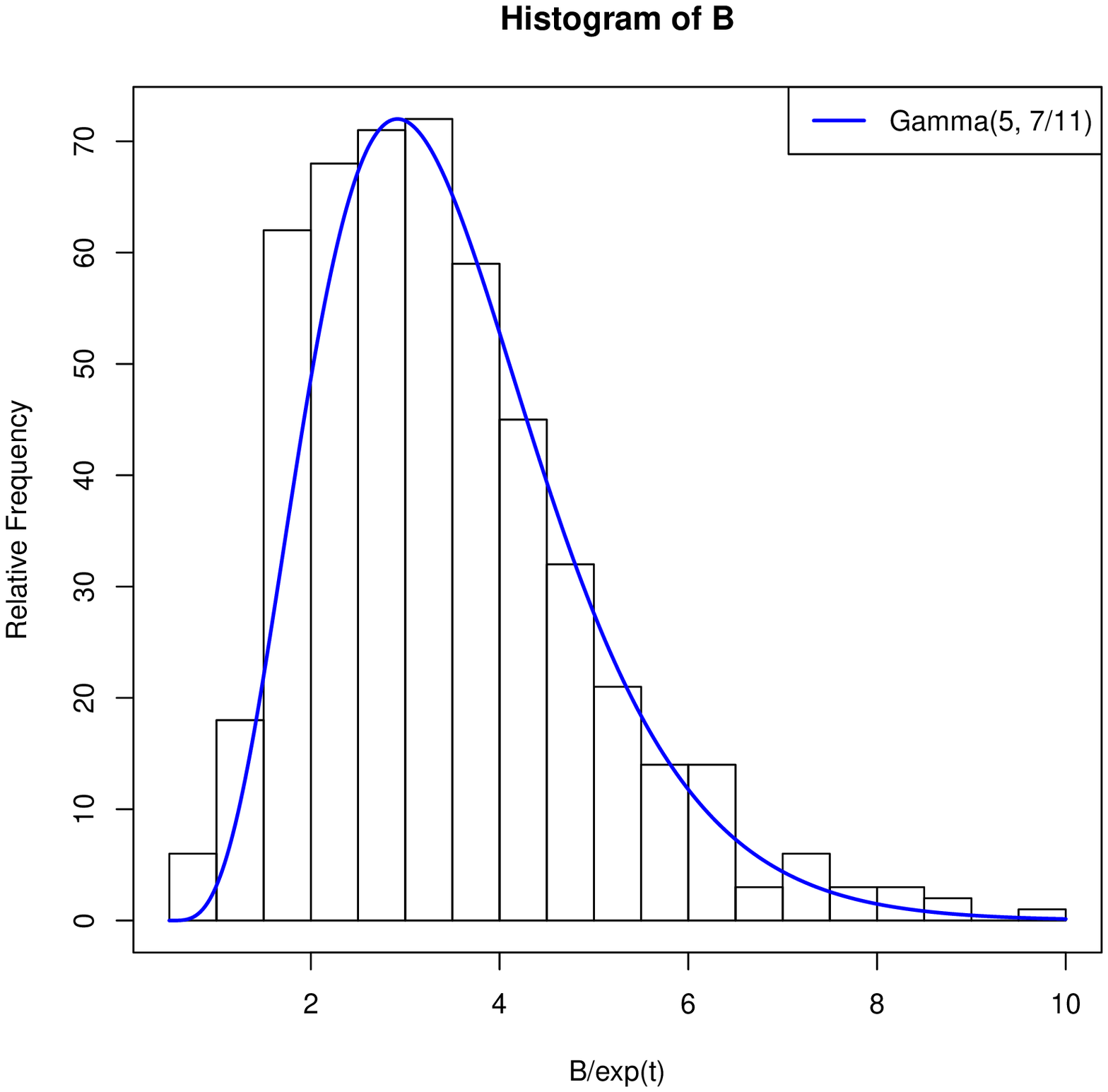}
		\caption{Histogram of $B(t)/e^{t}$ of the P\'{o}lya process with randomized replacement matrix.}
		\label{Fig:PWb}
	\end{minipage}
\end{figure}

We also compute the sample proportion of the number of white balls in the urn over all replicates, and get $0.36314$, cf.\ the theoretically asymptotic proportion $q_2 / (q_1 + q_2) \approx 0.36364$. In addition, the sample correlation between $W(t)/e^{t}$ and $B(t)/e^{t}$ is $0.99847$.

\section{Concluding remarks}
\label{Sec:concluding}
We finally add some concluding remarks in this section. In the discrete-time Bagchi-Pal urn schemes, the numbers of white balls $W_n$ and blue balls $B_n$ both asymptotically have normal distributions after properly scaled~\cite{Bagchi}, under an additional condition of $b + c \ge k/2$. By contrast, in the continuous-time Bagchi-Pal processes that arise in the Poissonized Bagchi-Pal urn schemes, we need to appropriately scale the corresponding random variables $W(t)$ and $B(t)$ to obtain limiting distributions, and both limits are gamma. We see that embedding into continuous time produces remarkably different results.

This note provides a novel perspective to characterizing the asymptotic behavior of the stochastic processes. Our approach is to determine asymptotic mixed moments of the process in a bootstrapped manner from the ordinary differential equations obtained by applying some differential operator to the PDE that governs the process. Noteworthy as a fact is that our approach herein is elementary, as it only requires basic knowledge of calculus (the {\em Leibniz rule}) and linear algebra (matrix differential equation). Our method may be of broad applicability in many other types of stochastic processes, which remains to be explored in our future work. We also would like to point out that the methods in~\cite{Janson2004} and~\cite{Chauvin} were restricted to some conditions (all entries in the replacement matrix at least $-1$ for the former and an ordinary balance condition for the latter). It seems that neither of these restrictions has effect on our approach, albeit some tenability issues need to be stressed. We would like to do some further investigations in the this direction in our future work as well.

\vskip 3mm

\noindent ACKNOWLEDGMENT

The authors would like to thank Professor Hosam M.\ Mahmoud for providing many valuable insights and giving many ingenious suggestions to this manuscript. We are also grateful to the anonymous referees for their helpful advice and comments.

%\bibliographystyle{chicago}
%\bibliography{polyaprocess}   % name your BibTeX data base

\end{document}